\documentclass[12pt]{article}

\usepackage{graphicx}
\usepackage{bm}
\usepackage[T1]{fontenc}
\usepackage{mathptmx}
\usepackage{cite}
\usepackage{amstext,amsthm,amsfonts}
\usepackage{amsmath}
\usepackage{enumerate}
\usepackage{array}
\usepackage{amsopn}
\usepackage{amsthm}
\usepackage{latexsym}
\usepackage{units}
\usepackage{mathrsfs}
\usepackage{geometry}
\usepackage[latin1]{inputenc}
\usepackage[english]{babel}
\usepackage{amssymb,etoolbox}
\usepackage{amsthm}

\def\lb{\lambda}
\def\Op{\mathfrak{Op}}
\def\H{\mathcal{H}}
\def\de{\mathrm{d}}

\def\A{{\mathcal A}}

\def\R{\mathbb R}
\def\C{\mathbb C}

\newtheorem{Theorem}{Theorem}[section]
\newtheorem{proposition}[Theorem]{Proposition}
\newtheorem{Remark}[Theorem]{Remark}
\newtheorem{Definition}[Theorem]{Definition}
\newtheorem{Lemma}[Theorem]{Lemma}
\newtheorem{Corollary}[Theorem]{Corollary}

\usepackage{color}

\title{Explicit Spectral Analysis for Operators Representing the unitary group $\mathbb{U}(d)$ and its Lie algebra $\mathfrak{u}(d)$ through the Metaplectic Representation and Weyl Quantization}
\author{F. Belmonte \,\,\,\, G. de Nittis}
\date{}

\begin{document}

\maketitle
\abstract{In this article we compute and analyze the spectrum of operators defined by the metaplectic representation $\mu$ on the unitary group $\mathbb{U}(d)$ and operators defined by the corresponding induced representation $d\mu$ of the Lie algebra $\mathfrak{u}(d)$. We will show that the point spectrum of both types of operators can be expressed in terms of the eigenvalues of the corresponding matrices. For each $A\in\mathfrak{u}(d)$, we will give conditions to guarantee that $H_A=-i d\mu(A)$ has discrete spectrum. Under these conditions, using a known result in combinatorics, we show that  the multiplicity of the eigenvalues of $H_A$ is (up to some explicit translation and scalar multiplication) a quasi-polynomial of degree $d-1$. Moreover, we show that the counting of eigenvalues function behaves as an Ehrhart polynomial. Using the latter result, we prove Weyl's law for the operators $H_A$.

 }
\section{Introduction}
Let $A=\big(\begin{smallmatrix} 
B & C\\
-C & B
\end{smallmatrix}\big)$, where $B$ and $C$ are $d\times d$ real matrices such that $B^*=-B$ and $C^*=C$. The main purpose of this article is to compute and analyse the spectrum of the following family of operators on $L^2(\R^d)$ with domain $S(\R^d)$ (i.e. the Schwartz space): 

$$
H_A=\frac{1}{2}\sum C_{jk}(-\frac{\partial^2}{\partial x_j\partial 
 x_k}+x_jx_k)+\frac{i}{2}\sum B_{jk}(x_k\frac{\partial}{\partial x_j}-x_j\frac{\partial}{\partial x_k}).
$$
It turns out that, under the identification $\R^{2d}\ni(x,\xi)\mapsto x+i\xi\in\C^d$, the matrices $A$ of the previously described type correspond with matrices belonging to the Lie algebra $\mathfrak{u}(d)$ of anti-Hermitian matrices and the map $A\mapsto H_A$ is a Lie algebra homomorphism. Recall that any anti-hermitian matrix has purely imaginary eigenvalues.  
\begin{Theorem}\label{spec}
Let $is_1,is_2,\cdots is_d$ be the eigenvalues of $A$, $s=(s_1,s_2,\cdots,s_d)$ and $L(s)=\{-\sum s_j n_j\mid n\in\mathbb N^d_0\}$. 
\begin{enumerate}
\item[i)] The point spectrum of $H_A$ is given by
$$
\sigma_p(H_A)=L(s)+\frac{i}{2}\text{tr}(A)
$$
and the spectrum of $H_A$ is $\sigma(H_A)=\overline{\sigma_p(H_A)}$. Let $\{v_j\}_{j=1}^d$ be a basis of $\C^d$ such that $Av_j=is_jv_j$, denote by $\{e_j\}$ the canonical basis of $\C^d$ and $\langle\cdot,\cdot\rangle$ its canonical inner product. For each $(n_1,n_2,\cdots, n_d)\in\mathbb{N}_0^d$, let $q$ be the polynomial on $\C^d$ given by
$$ 
q(z)=\left(\sum \langle v_{1}, e_j\rangle z_j\right)^{n_1}\cdots \left(\sum \langle v_{d}, e_j\rangle z_j\right)^{n_d}.
$$
Then $u=\hat B^*q$ is an eigenvector of $H_A$ with eigenvalue $\lb=-\sum s_j n_j+\frac{i}{2}\text{tr}(A)$, where $\hat B$ is the Bargmann transform. Every eigenvector is an orthogonal sum of finite linear combinations of vectors of the latter form.
\item[ii)] The following statements are equivalent:
\begin{enumerate}
    
    \item[a)] The subgroup of $\R$ generated by the monoid $L(s)$ is of the form $x\mathbb Z$, for some $x\in\R$.
    \item[b)] There is $x\in\R$ and $p_1,\cdots p_d\in\mathbb Z$ such that $s_j=p_j x$.
    \item[c)]  $\sigma_p(H_A)$ is uniformly topologically discrete, i.e. there is $r>0$ such that $(\lambda-r,\lambda+r)\cap (\zeta-r,\zeta+r)=\emptyset$, for every $\lambda,\zeta\in\sigma_p(H_A)$. 

\end{enumerate}
If any of the previous statements holds then $\sigma(H_A)=\sigma_p(H_A)$.
\end{enumerate}
 
\end{Theorem}

The proof of the previous theorem will rely on two important facts: i) $H_A$ is the infinitesimal generator of the one parameter group $U_t=\mu(e^{tA})$, where $\mu$ is the so called metaplectic representation (see chapter 4 in \cite{Fol}). ii) $H_A$ has an explicit Weyl symbol $p_A$ (see equation \eqref{pA}), which Poisson commutes with the classical harmonic oscillator, and therefore (according to theorems 1,2, 3 in \cite{BC}) $H_A$ strongly commutes with the quantum harmonic oscillator. In subsection \ref{Not} we explain in more detail these facts and others that we will need later.

Besides using the metaplectic representation to prove theorem \ref{spec}, we will also compute the spectrum of $\mu(g)$, for every $g\in \mathbb{U}(d)$.  

\begin{Theorem}\label{smu}
Let $\theta_1,\cdots ,\theta_d$ be the eigenvalues of $g\in \mathbb{U}(d)$. Then, the point spectrum of $\mu(g)$ is given by  
$$
\sigma_p(\mu(g))= \det(g)^{-1/2}\cdot\{\overline{\theta}_1^{n_1}\cdots\overline{\theta}_d^{n_d}\mid (n_1,\cdots, n_d)\in\mathbb N_0^d \}.
$$
Let $\{v_j\}_{j=1}^d$ be a basis of $\C^d$ such that $gv_j=\theta_jv_j$, denote by $\{e_j\}$ the canonical basis of $\C^d$ and $\langle\cdot,\cdot\rangle$ its canonical inner product. For each $(n_1,n_2,\cdots, n_d)\in\mathbb{N}_0^d$, let $q$ be the polynomial on $\C^d$ given by
$$ 
q(z)=\left(\sum \langle v_{1}, e_j\rangle z_j\right)^{n_1}\cdots \left(\sum \langle v_{d}, e_j\rangle z_j\right)^{n_d}.
$$
Then $u=\hat B^*q$ is an eigenvector of $\mu(g)$ with eigenvalue $\theta=\det(g)^{-1/2}\cdot\overline{\theta_1}^{n_1}\cdots \overline{\theta_d}^{n_d}$, where $\hat B$ is the Bargmann transform. For any $g\in \mathbb{U}(d)$, $g$ has an irrational rotation eigenvalue if and only if $\sigma(\mu(g))=\mathbb S^1$. Moreover, if $\theta_j=\text{exp}\left(\frac{2\pi ip_j}{q_j}\right)$ with $p_j\in\mathbb Z$ and $q_j\in\mathbb N$, then
$$
\sigma_p(\mu(g))=\sigma(\mu(g))= \det(g)^{-1/2}\cdot\left\{\text{exp}(\frac{2\pi inp}{q})\mid n\in\mathbb N \right\},
$$
where $q$ is the least common multiple of the denominators $q_1,\cdots, q_d$ and $p$ is the greatest common divisor of $\frac{q|p_1|}{q_1},\cdots, \frac{q|p_d|}{q_d}$. In particular, $\sigma(\mu(g))$ is a rotation of a finite subgroup of $\mathbb S^1$.
\end{Theorem}

The proofs of the previous theorems will be given in section \ref{Cspec}.

In section \ref{MyW}, we will provide conditions to guarantee that the spectrum of our operators are discrete and, under such conditions, we will study the corresponding counting of eigenvalues function. From the proof of theorem \ref{smu}, it will become clear that $\mu(g)$ never has discrete spectrum, because either the eigenvalues have infinite multiplicity or they are not isolated. Instead, if any of the conditions a), b) or c) in theorem \ref{spec} holds, $H_A$ has discrete spectrum if and only if all the eigenvalues of $-iA$ have the same sign. Moreover, if the eigenvalues of $-iA$ are $s_j=p_jx$ with $j=1,\cdots, d$, $p_j\in\mathbb Z$ and $x\in\R$, it will become clear that computing the multiplicity of $\lb\in\sigma(H_A)$ is equivalent to counting  how many $(n_1,\cdots,n_d)\in\mathbb N^d_0$ are such that $\lb=-\sum n_j s_j+i\text{tr}(A)/2$. Fortunately, such problems were studied long ago and the following result follows from well known facts on this topic (see theorem $2$ in \cite{Wr}, or  \cite{Bl} for a combinatorial proof of that theorem, without using generating functions).

\begin{proposition}\label{mul}
 Let $A\in\mathfrak u(d)$ and $is_1,is_2,\cdots,is_d$ its eigenvalues. Assume that $s_j=p_j x$ with $p_j\in\mathbb Z$ and $x\in\R-\{0\}$, for each $1\leq j\leq d$. All the eigenvalues of $H_A$ have finite multiplicity if and only if the real numbers $s_1,\cdots,s_d$ have the same sign. In such case, if $m_A(\lb)$ denotes the multiplicity of the eigenvalue $\lb$, then
 
 $$
m_A(\lambda)=\sum_{j=1}^d a_j(|(\lb-\frac{i}{2}\text{tr}(A)) x^{-1}|)\lb^{j-1},
$$
where $a_j(k)$ depends only of residues of $k$ moduli $d!$, for each $k\in\mathbb N$ and $1\leq j\leq d$. 

\end{proposition}
As a consequence of the previous result, if the eigenvalues $-s_1,-s_2,\cdots,-s_d$  of $iA$ are all positive and $s_j=p_j x$ with $p_j\in\mathbb Z$ and $x\in\R-\{0\}$,, we can define the counting of eigenvalues function $N_A(r):=\#\{\lb\in \sigma(H_A)\mid\lambda\leq r\}$. Once again we will use some combinatorial tools to analyze the map $N$. Indeed, using a so called Ehrhart polynomial (i.e. a polynomial of the form given by equation \eqref{Eh}), we will obtain the following result.

\begin{Theorem}[spectral asymptotics]\label{dis}
Let $is_1,is_2,\cdots is_d$ be the eigenvalues of $A\in\mathfrak u(d)$. Assume that $s_j<0$ and that there are $p_j\in\mathbb{Z}$ and $x\in\R$ such that $s_j=p_j x$, for each $1\leq j\leq d$. Choose $x>0$ and let $q$ be the minimal common multiple of $-p_1,-p_2,\cdots -p_d$. Then there is a polynomial $p(k)=\sum_{j=0}^d c_j k^j$ such that 
$$
p\left([(r-\frac{i}{2}\text{tr}(A))(qx)^{-1}]\right)\leq N_A(r)\leq p\left([(r-\frac{i}{2}\text{tr}(A))(qx)^{-1}]+1\right),
$$
where $[t]$ denotes the integer part of $t$, for any $t\in\R$. The inequality in the left hand side becomes an equality whenever $(r-\frac{i}{2}\text{tr}(A))(qx)^{-1}\in\mathbb N_0$. Moreover, if $\mathcal P=\{x\in\R^d\mid x\geq 0, -\sum x_jp_j\leq q \}$, then $c_0=1$, $c_d=|\mathcal P|$ is the volume of $\mathcal P$ and $c_{d-1}=\frac{1}{2} |\partial\mathcal P|$ is  one half of the sum of the $(d-1)$-volume of the faces of $\mathcal P$.  
\end{Theorem}

Since the operator $H_A$ comes from Weyl quantization, we can introduce Planck's constant dependence in our framework and we can study some of the emerging semiclassical problems. In other words, we shall consider the operators $H^\hslash_A:=\Op^\hslash(p_A)$ and analyze what happens when $\hslash\to 0$, where $\Op^\hslash$ denotes the $\hslash$- dependent Weyl quantization (see subsection \ref{WQ} for details). It turns out that $H^\hslash_A$ is unitary equivalent to $\hslash H_A$. Therefore, $\sigma_p(H^\hslash_A)=\hslash(L(s)+i\text{tr}(A)/2)$ and the spectral analysis provided in theorems \ref{spec} and \ref{dis} and proposition \ref{mul} also holds. In particular, the same polynomial behavior occurs for the counting of eigenvalues function $N_A^\hslash$ of $H^\hslash_A$ and this allowed us to prove the following semiclassical result (usually called Weyl's law in the literature). 

\begin{Theorem}[Weyl's law]\label{WL}
Let $is_1,is_2,\cdots is_d$ be the eigenvalues of $A\in\mathfrak u(d)$ and $s=(s_1,\cdots, s_d)$. Assume that $s_j<0$ and that there are $p_j\in\mathbb{Z}$ and $x\in\R$ such that $s_j=p_j x$, for each $1\leq j\leq d$. Also let $H_A^\hslash=\Op^\hslash(p_A)$, $N^\hslash_A(r)=\#\{\lb\in\sigma(H_A^\hslash)\mid \lb\leq r\}$ and $\mathcal E_A(r)=\{(x,\xi)\in\R^{2d}\mid p_A(x,\xi)\leq r\}$. Then
$$
N^\hslash_A(r)=(2\pi\hslash)^{-d}|\mathcal E_A(r)|+\frac{\|s\|}{2}(2\pi)^{-d}
\hslash^{-d+1}\int_{\partial\mathcal E_A(r)}\|\nabla p_A\|^{-1}d\mu_r^{A}+O\left(
\left(\frac{r}{\hslash}\right)^{d-2}\right).
$$
where $\mu_r^{A}$ is the measure corresponding to the canonical volume form on  $\partial\mathcal E_A(r)=\{(x,\xi)\in\R^{2d}\mid p_A(x,\xi)=r\}$.
\end{Theorem}

Weyl's law is an important topic in operator theory. For a quite complete review on this topic, we recommend \cite{Iv}. It is mentioned there that, for the Laplace operator over a rectangular box (with Dirichlet or Neumann boundary conditions), Weyl's law is obtained by counting the non-negative integer lattice points inside certain inflated ellipsoid. Instead, Ehrhart polynomials are defined counting the non-negative integer lattice points inside an inflated polyhedra. However, during the proof of theorem \ref{WL}, we ended up relating the inflated polyhedra $k\mathcal P$ with the inflated ellipsoid  $\mathcal E(r)$ (with $k=\left[(r-\frac{i}{2}\text{tr}(A))(qx)^{-1}\right]$).

We should also mention that the reminder of the Weyl's law with two terms for large classes of operators is usually of the form $o\left((\frac{r}{\hslash})^{d-1}\right)$ (see for instance, \cite{Iv, DGW}). Of course, the improvement in the remainder of our result is due to the particular type of operators we are considering. 

In relation to the Harmonic Oscillator, we recommend \cite{DGW,Zw}. In \cite{DGW} Weyl's law was obtained for perturbations of the Harmonic operator by an isotropic pseudodifferential operator of order 1. We wonder if the techniques applied there can be adapted to extend their results to perturbations of the operators $H_A$. In \cite{Zw} Weyl's law for the Harmonic oscillator itself is given (theorem 6.3) and also for Schr\"{o}dinger operators with suitable potential (theorem 6.8).

\medskip
 
 \noindent
{\bf Acknowledgements.}
GD's research is supported by the grant \emph{Fondecyt Regular - 1230032}.

\section{Preliminaries}\label{Not}

\subsection{Weyl Quantization and the Metaplectic Representation.}\label{WQ}
The operators that we will study in this article come either from Weyl quantization or the metaplectic representation, so we decided to include a brief introduction of these objects in this subsection, including the features of them that we will need later.
   
Weyl quantization \cite{Wey, Gro} (or Weyl calculus or canonical quantization) is a map meant to transform real functions on the canonical phase space $\R^{2d}$ (classical observables) into selfadjoint operators on $L^2(\R^d)$ (quantum observables) in a physically meaningful manner. Formally, for $\hslash>0$ and $f\in S'(\R^{2d})$, we define  $\Op^\hslash(f):S(\R^d)\to S'(\R^d)$ by

\begin{equation}\label{We}
[\Op^\hslash(f)u](x)=(2\pi\hslash)^{-d}\int_{\R^d}\int_{\R^d}f\left(\frac{x+y}{2},\xi\right)e^{\frac{i}{\hslash}(x-y)\cdot \xi}u(y)\de\xi\de y, 
\end{equation}
where $S'(\R^m)$ denotes the topological dual of the Schwartz space $S(\R^m)$ (i.e. the space of tempered distributions). Notice that the integral above in general makes sense only as a tempered distribution. Indeed, $\Op^\hslash(f)$ has a kernel $K_f$ defined by composing a partial Fourier transform and a linear change of variable over $f$, thus $K_f$ belongs to $S'(\R^{2d})$. Therefore, $\Op^\hslash(f)$ is  defined by $[\Op(f)u](v)= K_f(v\otimes u)$, for every $u,v\in S(\R^d)$. 

There are several approaches to introduce and study Weyl quantization (for instance, see \cite{Zw,Tay, Shu}), but for the purposes of this article we will mainly follow \cite{Fol}, where a more group theoretic approach is used. However, there is a minor difference with our definition. Indeed, if we denote by $\widehat{\Op}^\hslash$ the quantization defined in \cite{Fol}, then $\widehat{\Op}^{2\pi\hslash}=\Op^{\hslash}$. This lead to some superficial changes in some formulas, but whenever this need to be considered we will recall it.

One of the main properties of Weyl quantization is its relation with the so called metaplectic representation. All the results mentioned below concerning this topic can be found in \cite[Chapter 4]{Fol}. Let $Sp(d)$ be the real symplectic group, i.e. the group formed by all the linear and symplectic maps $S:\R^{2d}\to\R^{2d}$. Let $\mathcal U(L^2(\R^d))$ denotes the group of unitary operators over $L^2(\R^d)$. The metaplectic representation is a map $\mu_\hslash: Sp(d)\to\mathcal U(L^2(\R^d))$ such that
\begin{equation}\label{meta}
\Op^\hslash(f\circ S^*)=\mu_\hslash(S)\Op^\hslash(f)\mu_\hslash(S)^{-1},
\end{equation}
for every $S\in Sp(d)$ and $f\in S'(\R^{2d})$. Observe that  \cite[Theorem 2.15]{Fol} concerns the case $\hslash=(2\pi)^{-1}$, and the general case follows using the same proof. 

Once again, we have a minor difference between the way the metaplectic representation is defined in \cite{Fol} and ours. Since $\widehat{\Op}^{2\pi\hslash}=\Op^{\hslash}$, if one denotes by $\hat\mu_\hslash$ the metaplectic representation defined in \cite{Fol}, then $\mu_\hslash=\hat\mu_{2\pi\hslash}$.

Notice that equation \eqref{meta} makes sense because $\mu_\hslash(S)$ is an isomorphism from $S(\R^d)$ into itself, and it extends continuously to an isomorphism from $S'(\R^d)$ into itself  \cite[Proposition 4.27]{Fol}. 
 
Some authors call equation \eqref{meta} the exact Egorov theorem because in a certain sense Egorov theorem asserts a similar claim for arbitrary symplectomorphisms (not necessarily linear)  but only in the semiclassical limit, i.e. when $\hslash\to 0$  \cite[Theorems 11.1, 11.9 \& 11.10]{Zw}.  

Usually, we consider the operators coming from Weyl quantization and the metaplectic representation for the fixed value $\hslash=1$. In this case, we define $\Op=\Op^{1}$ and $\mu=\mu_1$. 

The unitary representations $\{\mu_\hslash\}_{\hslash>0}$ are all unitary equivalent. Indeed,
for each $g\in Sp(d)$, we know that (see the proof of  \cite[Theorem 4.57]{Fol})
\begin{equation}\label{equi-mu}
\mu_\hslash(g)=\mu(D_\hslash)\mu(g)\mu(D_\hslash)^*
\end{equation}
where $D_\hslash(x,\xi)=(\hslash^{-1/2}x,\hslash^{1/2}\xi)$. Furthermore, $\mu(D_\hslash)=\mu_\hslash(D_\hslash)=:\mathcal D_\hslash$ is the $\hslash^{1/2}$- dilation operator, i.e. $[\mathcal D_\hslash\varphi](x)=\hslash^{\frac{d}{4}}\varphi(\hslash^{1/2}x)$ \cite[eq. (4.24)]{Fol}.

Despite its name, $\mu_\hslash$ is not a representation of $Sp(d)$ in the usual sense, because in general, it satisfies the relation $\mu_\hslash(ST)=\pm\mu_\hslash(S)\mu_\hslash(T)$. However, this issue vanishes when $\mu_\hslash$ is restricted to the maximal compact subgroup $\mathbb{O}(2d)\cap Sp(d)$, where $\mathbb{O}(2d)$ is the orthogonal group on $\R^{2d}$. It will become very important later to notice that if we identify $\R^{2d}$ with $\C^d$ via the map $(x,\xi)\mapsto x+i\xi$, then $\mathbb{O}(2d)\cap Sp(d)=\mathbb{U}(d)$ is the complex unitary group on $\C^d$ \cite[Proposition 4.6]{Fol}.

Let us consider the symplectic Lie algebra $\mathfrak{sp}(d)$  corresponding to the Lie group $Sp(d)$. It turns out that $\mathfrak{sp}(d)$ coincides with the space of $2d\times 2d$ real matrices $A$ such that the matrix $ A\mathfrak{J}$ is symmetric, where $\mathfrak{J}$ is the so called symplectic matrix, i.e. $\mathfrak{J}=\big(\begin{smallmatrix} 
0 & I\\
-I & 0
\end{smallmatrix}\big)$ \cite[Proposition 4.2]{Fol}. Such matrix $A$ has associated the quadratic homogeneous polynomial $p_A$ on $\R^{2d}$ given by
\begin{equation}\label{pA}
p_A(w)=-\frac{1}{2}w\cdot A\mathfrak{J}\cdot w,
\end{equation}

From  \cite[Proposition 4.2]{Fol} one knows that if $A$ belongs to the symplectic Lie algebra then $A=\big(\begin{smallmatrix} 
B & C\\
D & -B^*
\end{smallmatrix}\big)$, with $D=D^*$ and $C=C^*$. Therefore 
$$
p_A(x,\xi)=\frac{1}{2}\langle x,Cx\rangle-\langle x,B\xi\rangle-\frac{1}{2}\langle \xi,D\xi\rangle.
$$
Since $\Op^\hslash(x_j\xi_k)=-\frac{i\hslash}{2}(x_j\frac{\partial}{\partial x_k}+x_k\frac{\partial}{\partial x_j})$, we have that
\begin{equation}\label{oppa}
\Op^\hslash(p_A)=\frac{\hslash^2}{2}\sum D_{ij}\frac{\partial^2}{\partial x_j\partial 
 x_k}+i\hslash\sum B_{jk}x_k\frac{\partial}{\partial x_j}+\frac{1}{2} \sum C_{jk}x_jx_k-\frac{\hslash}{2}\text{tr}(B).
\end{equation}

Moreover, if $A\in\mathfrak{u}(d)$, under the identification $\R^{2d}\ni (x,\xi)\mapsto x+i\xi\in \C^d$, we have  $A=\big(\begin{smallmatrix} 
B & C\\
-C & B
\end{smallmatrix}\big)$, with $B^*=-B$ and $C=C^*$ (as in the introduction). Hence, equation \eqref{oppa} becomes 
$$
\Op^\hslash(p_A)=-\frac{\hslash^2}{2}\sum C_{jk}\frac{\partial^2}{\partial x_j\partial 
 x_k}+i\hslash\sum_{j<k} B_{jk}(x_k\frac{\partial}{\partial x_j}-x_j\frac{\partial}{\partial x_k})+\frac{1}{2} \sum C_{jk}x_jx_k.
$$

The operators $\Op^\hslash(p_A)$ can be obtained directly from the metaplectic representation. Indeed, let $d\mu_\hslash$ be the infinitesimal representation associated to $\mu_\hslash$ on $S(\R^d)$, i.e. for $A\in\mathfrak{sp}(d)$ and $\varphi\in S(\R^d)$ define 

$$
d\mu_\hslash(A)\varphi=\frac{d}{dt}[\mu_\hslash(e^{tA})\varphi]|_{t=0}.
$$
Notice that \eqref{equi-mu} implies that 
\begin{equation}\label{equi-dmu}
d\mu_\hslash(A)=\mu(D_\hslash)d\mu(A)\mu(D_\hslash)^*.
\end{equation}
\begin{proposition}\label{hb-H_A}
For each $A\in\mathfrak{sp}(d)$ and $\hslash>0$, we have that $-i\hslash\, d\mu_\hslash(A)=\Op^\hslash(p_A)$. 
\end{proposition}
\begin{proof}
This result for $\hslash=(2\pi)^{-1}$ is proved in \cite[Theorem 4.45]{Fol}. Moreover, \eqref{equi-dmu} and \eqref{meta} imply that
\begin{align*}
-id\mu_\hslash(A)=\mu(D_{\hat\hslash})[-id\mu_{(2\pi)^{-1}}(A)]\mu(D_{\hat\hslash})^*= 2\pi\mu(D_{\hat\hslash})&\Op^{\frac{1}{2\pi}}(p_A)\mu(D_{\hat\hslash})^*\\
&=2\pi\Op^{\frac{1}{2\pi}}(p_A\circ D_{\hat\hslash}^*),
\end{align*}
where $\hat\hslash=2\pi\hslash$. Let us compute the right hand side of the previous identity.
$$
2\pi[\Op^{\frac{1}{2\pi}}(p_A\circ D_\hslash)u](x)=2\pi\int_{\R^d}\int_{\R^d}p_A\left(\frac{\hslash^{-1/2}(x+y)}{2},\hslash^{1/2}\xi\right)e^{2\pi i(x-y)\xi}u(y)\,d\xi dy
$$
Performing the change of variables $\xi'=\hat\hslash\xi$ and using the definition of the polynomial $p_A$, we obtain
\begin{align*}
&2\pi [\Op^{\frac{1}{2\pi}}(p_A\circ D_{\hat\hslash})u](x)\\
=&2\pi(\hat\hslash)^{-d}\int_{\R^d}\int_{\R^d}p_A\left(\frac{\hat\hslash^{-1/2}(x+y)}{2},\hat\hslash^{-1/2}\xi\right)e&^{i\hslash^{-1}(x-y)\xi}u(y)\,d\xi dy=\hslash^{-1}\Op^\hslash(p_A)
\end{align*}

and this finishes the proof.
\end{proof}

Computing explicitly  $\mu_\hslash(g)$ might be difficult in general. One way to do it is, up to a sign, to use the known values of $\mu_{(2\pi)^{-1}}$ on the generators of $Sp(d)$ given by \cite[eq. (4.24), (4.25) \& (4.26)]{Fol}. Another way is to give another representation $\nu$ explicitly, and then show it is unitary equivalent with $\mu_\hslash$. Let us begin with the case $\hslash=(2\pi)^{-1}$ described in \cite{Fol} and references therein. Let $\mathfrak{F}^d$ be the Segal-Bargmann space (some authors call it the Fock space), i.e. $\mathfrak{F}^d$ is the Hilbert space formed by all the holomorphic functions $F:\C^d\to\C$ such that $\int |F(z)|^2 e^{-\pi|z|^2}\, dz<\infty$. Also, let $\hat B:L^2(\R^d)\to\mathfrak{F}^d$ be the Bargmann transform  \cite[Chapter I, Section 6]{Fol}. It turns out that     $\nu(g):=\hat B\mu_{(2\pi)^{-1}}(g)\hat B^*$ admits an explicit expression as an operator with kernel (see \cite[Theorem 4.37]{Fol}). Moreover, \cite[Proposition 4.39]{Fol} asserts that, if $g\in \mathbb{U}(d)$ then 
\begin{equation}\label{nu}
[\nu(g)q](z)=\det(g)^{-1/2}q(g^{-1}z).  
\end{equation}
From \eqref{equi-mu}, it is clear that every $\mu_\hslash$ is unitary equivalent with $\nu$. Indeed, if we define $B_\hslash=\hat B \mathcal D_{(2\pi \hslash)^{-1}}$, then  
$$
 B_{\hslash}\mu_{\hslash}(g) B_{\hslash}^*=\nu(g).
$$

\subsection{The Harmonic Oscillator.}\label{HO}
This  section is meant to summarize some  well-known facts concerning the \emph{harmonic oscillator} and some results from \cite{BC} that we are going to need later.

The classical harmonic oscillator is the classical Hamiltonian defined by $h_0(x,\xi)=\frac{1}{2}(\|x\|^2+\|\xi\|^2)$. It is easy to check that the classical flow of $h_0$ is given by
$$\varphi_t (x,\xi)=\begin{bmatrix}
(\cos t)I&(\sin t)I\\
(-\sin t)I&(\cos t)I
\end{bmatrix}\begin{bmatrix}
x\\
\xi
\end{bmatrix}.$$
Under the identification of $\mathbb{R}^{2d}\ni (x,\xi)\mapsto x+i\xi\in\mathbb{C}^{d}$, the flow admits the representation 
$$
\varphi_t(x+i\xi)=e^{-it}(x+i\xi).
$$

In particular, $\varphi_t$ belongs to the center of $\mathbb{U}(d)$ i.e. it commutes with every element of $\mathbb{U}(d)$, for each $t\in\R$. 

We say that $f\in C^\infty(\R^{2d})$ is a classical constant of motion for $h_0$ if $\{h_0,f\} = 0$, where $\{\cdot,\cdot\}$ denotes the Poisson bracket corresponding to the canonical symplectic structure on $\R^{2d}$. Leibniz's rule and Jacobi identity show that the set $\A$ of all constants of motion is a Poisson subalgebra of $C^\infty(\R^{2d})$. It is easy to show that $f$ belongs to $\A$ if and only if $f\circ\varphi_t=f$, for each $t\in\R$. Since $\varphi_t$ is linear and preserves volume, if $f\in S(\R^{2d})$ then $f\circ\varphi_t\in S(\R^{2d})$ and this composition can be extended to $S'(\R^{2d})$, as in the following definition.

\begin{Definition}\label{TCOM}
For each $T\in S'(\mathbb{R}^{2d})$, we define the distribution $\varphi^*_t T$ by  
$$
(\varphi^*_t T)(f)=T(f\circ \varphi_t),
$$ 
where $f$ is any rapidly decreasing function on $\R^{2d}$. If $\varphi^*_t T=T$, we say that $T$ is tempered constant of motion of the Harmonic Oscillator.
\end{Definition}

The $d$-dimensional quantum harmonic oscillator is the self-adjoint operator $H_0=\frac{1}{2}(-\Delta+\|x\|^{2})$ on $L^2(\mathbb{R}^{d})$. The spectrum of $H_0$ is discrete and equal to $\mathbb N_0+\frac{d}{2}$. Each eigenvalue $k+\frac{d}{2}$ has multiplicity $d_k=\binom{d+k-1}{k}$. We denote by $\mathcal{H}_k$ the corresponding eigenspace and we denote by $P_k$ the corresponding orthogonal projection. It turns out that the $d$-dimensional Hermite functions $\{\phi_\alpha\mid \alpha\in\mathbb{N}^d_0 \}$ form an orthonormal basis of eigenvectors, where
$$
\phi_{\alpha}(x)=(-1)^{|\alpha|}\pi^{-d/4}(2^{|\alpha|}\alpha!)^{-1/2}e^{\frac{\|x\|^{2}}{2}}\left(D^\alpha e^{-\|x\|^{2}}\right),
$$
with $|\alpha|=\alpha_1+\cdots\alpha_d$ and $\alpha !=\alpha_1 !\alpha_2!\cdots\alpha_d !$. Moreover,  $\{\phi_\alpha\mid \alpha\in\mathbb{N}^d_0, |\alpha|=k \}$ is a basis of $\H_k$, for each $k\in\mathbb{N}_0$. Thus,
$$L^2(\mathbb{R}^{d})=\bigoplus_{k\in\mathbb{N}_0}\mathcal{H}_{k}.$$

Notice that $\Op(h_0)=H_0$ i.e. Weyl quantization maps the classical to the  quantum harmonic oscillator. The following result is the combination of  \cite[Theorems 1, 2, 3 and Corollary 3]{BC} and, in a certain sense it asserts that Weyl quantization maps classical to quantum constants of motion of the harmonic oscillator.

\begin{Theorem}\label{PCM}
Let $f$ be a real tempered constant of motion of the classical harmonic oscillator. Then $\Op(f)[S(\R^d)]\subseteq S(\R^d)$. Moreover, $\Op(f)$ with domain $S(\R^d)$ is an essentially selfadjoint operator and we also denote by $\Op(f)$ its closure. Furthermore, $\Op(f)$
 strongly commutes with $H_0$. In particular, if we define $\Op_k(f):=\Op(f)|_{\H_k}$, then 
$$
\sigma_p\left(\Op(f)\right)=\bigcup_{k\in\mathbb N_0} \sigma_p\left(\Op_k(f)\right)
$$
and $\overline{\sigma_p\left(\Op(f)\right)}=\sigma\left(\Op(f)\right)$, where $\sigma$ and $\sigma_p$ denote the spectrum and the point spectrum respectively.
\end{Theorem}

We shall not provide details of the proof of the previous theorem here, but we would like to mention that a key ingredient to do so is to use that 

\begin{equation}\label{mu(h0)}
\mu(\varphi_t)=e^{itH_0}.
\end{equation}

Indeed, this identity and equation \eqref{meta} imply that if $f$ is a tempered constant of motion, then $\Op(f)=e^{itH_0}\Op(f)e^{-itH_0}$, which formally says that $\Op(f)$ commutes with $H_0$.
 
It is straightforward to show that, if $A\in\mathfrak{u}(d)\subset\mathfrak{sp}(d)$, then $p_A$ is a tempered constant of motion. In particular, the spectral decomposition given in theorem \ref{PCM} holds for $H_A=\Op(p_A)$, and this will allow us to compute the spectra of those operators.   

Restricting the metaplectic representation to $\mathbb{U}(d)$ also leads to operators strongly commuting with $H_0$. Indeed, as we mentioned before $\varphi_t$ commutes with every $g\in \mathbb{U}(d)$. Applying the metaplectic representation and \eqref{mu(h0)} we obtain that $\mu(g)e^{itH_0}=e^{itH_0}\mu(g)$ ( corollary 4 in \cite{BC}). Thus,

$$
\langle \mu(g)\phi_\alpha,\phi_\beta\rangle=\langle e^{itH_0}\mu(g)e^{-itH_0}\phi_\alpha,\phi_\beta\rangle=e^{it(|\alpha|-|\beta|)}\langle \mu(g)\phi_\alpha,\phi_\beta\rangle.
$$
Then $\langle \mu(g)\phi_\alpha,\phi_\beta\rangle=0$, unless $|\alpha|=|\beta|$. Hence, each $\H_k$ is invariant by $\mu(g)$. Therefore, if we denote by $\mu_k(g):=\mu(g)|_{\H_k}$, we obtain
\begin{equation}\label{dmu}
 \mu(g)=\bigoplus_{k}\mu_k(g), 
\end{equation}
for every $g\in \mathbb{U}(d)$. Derivating the previous identity and applying Theorem \ref{PCM} we obtain the following result \cite[Corollary 5]{BC}. 

\begin{Corollary}\label{metdecom1}
For each $A\in \mathfrak{u}(d)$, the operator $H_A=\Op(p_A)=-i d\mu(A)$ with domain $S(\R^d)$ is essentially selfadjoint and it strongly commutes with $H_0$. Moreover, the map $\de\mu$ admits the decomposition
$$
\de\mu =\bigoplus_{k}\de\mu_k,
$$
where $\mu_k(g):=\mu(g)|_{\H_k}$, for each $g\in\mathbb{U}(d)$ and $k\in\mathbb N_0$.
\end{Corollary}

\section{Computing Spectra}\label{Cspec}

Abusing of the notation, we will denote by $\mu$  the restriction of the metaplectic representation to the complex unitary group $\mathbb{U}(d)$. According to identity \eqref{dmu}, in order to compute the (point) spectrum of $\mu(g)$, it is enough to compute the spectrum of $\mu_k(g)$, for each $k\in\mathbb{N}_0$ and $g\in \mathbb{U}(d)$. Equivalently, we might compute the spectrum of $\nu_k(g)=B\mu_k(g) B^*=\nu(g)|_{\mathfrak{F}^d_k}$, where $\mathfrak F^d_k:=B(\H_k)$.  Besides $\nu_k$, we will need another unitary equivalent representation of $\mu_k$. It is well known that $\mathfrak F^d_k$ is the subspace generated by the monomials of total degree $k$. Let $S^k(\C^d)$ denotes the $k$-th symmetric power of $\C^d$ and $\{e_j\}$ be the canonical basis of $\C^d$. For each  $j_1,\cdots j_k\in\{1,\cdots , d\}$,  let $q_{j_1,\cdots j_k}(z)=z_{j_1}\cdots z_{j_k}$.  It is well know that the map $T_k:S^k(\C^d)\to \mathfrak F^d_k$ given by

$$
T_k(e_{j_1}\odot\cdots\odot e_{j_k})=q_{j_1,\cdots j_k}
$$
extends to an isomorphism of vector spaces, where $\odot$ denotes the symmetric tensor product. It is straightforward to check that
\begin{equation}\label{equi} 
T_k(v_1\odot\cdots\odot v_k)=T_1(v_1)\cdots T_1(v_k).   
\end{equation}

Moreover,  $T_1(v)=T_1(\sum \langle e_j,v\rangle e_j)=\sum \langle e_j,v\rangle z_j$, thus 
$$
[T_1(v)](z)=\langle \overline{v},z\rangle.
$$

The following result is apparently well known, but we could not find it stated, nor proved, in the literature. It will be the main tool to prove theorems  \ref{spec} and \ref{smu}.

\begin{proposition}\label{peta}
Let $\eta_k(g):=T_k^{-1}\nu_k(g) T_k:S^k(\C^d)\to S^k(\C^d)$. Then 
\begin{equation}\label{etak}
 \eta_k(g)(v_1\odot v_2\odot\cdots\odot v_k)=\det(g)^{-1/2}(\overline{g} v_1)\odot (\overline{g} v_2)\odot\cdots\odot (\overline{g} v_k),
\end{equation}
for any $g\in \mathbb{U}(d)$ and $v_1,v_2,\cdots v_k\in\C^d$.
\end{proposition}
\begin{proof}
Clearly we have that 
$$
[T_1(\overline{g}v)](z)=\langle \overline{\overline{g}v},z\rangle=\langle g\overline{v},z\rangle =\langle\overline{v}, g^{-1}z\rangle=\det(g)^{1/2}[\nu_1(g) T_1(v)](z)
$$
and this equivalent to equation \eqref{etak} for $k=1$. The general case follows from equation \eqref{equi}. 
\end{proof}

\begin{proof}[Proof of Theorem \ref{smu}]
Let $v_1,\cdots v_d$ be a basis of eigenectors of $g$ corresponding to $\theta_1,\cdots,\theta_d$. Thus, 
\begin{align}\label{uevec}
  \eta_k(g)(\overline{v_{j_1}}\odot\cdots\odot \overline{v_{j_k}})&=\det(g)^{-1/2}(\overline{g v_{j_1}})\odot\cdots\odot (\overline{g v_{j_k}})\\
  &=\det(g)^{-1/2}\cdot\overline{\theta}_{j_1}\cdots\overline{\theta}_{j_k}(\overline{v_{j_1}}\odot\cdots\odot \overline{v_{j_k}}) 
\end{align}

This together with equation \eqref{equi} shows that the polynomial $q$ given in theorem \ref{smu} is an eigenvector of $\nu(g)$ with eigenvalue $\theta=\det(g)^{-1/2}\cdot\overline{\theta_1}^{n_1}\cdots \overline{\theta_d}^{n_d}$, and therefore $\hat B^* q$ is an eigenvector of $\mu(g)$ with the same eigenvalue.

Since the collection of vectors $\{\overline{v_{j_1}}\odot\cdots\odot \overline{v_{j_k}}\}$ with $j_i\in\{1,\cdots, d\}$ forms a basis of $S^k(\C^d)$, we have that
$$
\sigma(\mu_k(g))=\sigma(\eta_k(g))= \det(g)^{-1/2}\{\overline{\theta}_{j_1}\cdots\overline{\theta}_{j_k}\mid j_i\in\{1,\cdots, d\}\}. 
$$
Hence, the decomposition of $\mu$ given by \eqref{dmu} implies that  
$$
\det(g)^{-1/2}\cdot\{\overline{\theta}_1^{n_1}\cdots\overline{\theta}_d^{n_d}\mid (n_1,\cdots, n_d)\in\mathbb N_0^d\}\subseteq \sigma_p(\mu(g)).
$$ 
Conversely, if $\mu(g)\varphi=\rho\varphi$, for some $\varphi\neq 0$ and $\rho\in\mathbb S^1$, then $\mu_k(g)P_k\varphi=\rho P_k\varphi$ for every $k\in\mathbb N$. Since $\varphi\neq 0$, $P_k\varphi\neq 0$ for some $k\in\mathbb N$, thus $\rho\in \sigma(\mu_k(g))$.

For the second claim, it is well known that every irrational rotation generates a dense orbit in $\mathbb S^1$. The converse will follow from proving the last claim of our result (concerning the case of rational rotations eigenvalues). In that case, let $Z(\theta)$ be the additive subgroup of $\R$ generated by $\frac{p_1}{q_1},\cdots, \frac{p_d}{q_d}$. Then multiplying by $q$ maps $Z(\theta)$ into the subgroup of $\mathbb Z$ generated by $\frac{q|p_1|}{q_1},\cdots, \frac{q|p_d|}{q_d}$. It is well known that such subgroup is generated by the greatest common divisor of $\frac{q|p_1|}{q_1},\cdots, \frac{q|p_d|}{q_d}$. In particular, $Z(\theta)=\frac{p}{q}\mathbb Z$ and this implies our result. 

\end{proof}
\begin{Remark}
{\rm
When $g=\mathfrak J$ is the canonical symplectic matrix, we have that $\theta_1=\cdots,=\theta_d=-i$ and $\mu(g)=i^{n/2}\mathcal F^{-1}$, where $\mathcal F$ is the Fourier transform (see \cite[eq. (4.26) and Proposition 4.46]{Fol}).  Theorem \ref{smu} implies the well known identity $\sigma(\mu(g))=\{1,-1,i,-i\}$.  
}    
\end{Remark}

We will use the same technique to compute the spectrum of $-id\mu_k(A)=\Op_k(p_A)=H_A|_{\H_k}$, for each $A\in \mathfrak{u}(d)$ and this will lead us to prove theorem \ref{spec}.

\begin{proof}[Proof of theorem \ref{spec}]
Taking $g=e^{tA}$ in equation \eqref{etak} and differentiating at $t=0$ we obtain that
\begin{align*}
d\eta_k(A)(v_1\odot v_2\odot\cdots\odot v_k)&=-\frac{1}{2}\text{tr}(A)(v_1\odot v_2\odot\cdots\odot v_k)
+(\overline{A}v_1)\odot v_2\odot\cdots\odot v_k\\
&+v_1\odot (\overline{A}v_2)\odot\cdots\odot v_k+\cdots +v_1\odot v_2\odot\cdots\odot(\overline{A}v_k) 
\end{align*}
Let $\{v_j\}$ be a basis of $\C^d$ such that $Av_j=is_jv_j$. Then 
\begin{align}
&d\eta_k(A)(\overline{v_{j_1}}\odot\cdots\odot\overline{v_{j_k}})\nonumber\\
&=-\frac{1}{2}\text{tr}(A)(\overline{v_{j_1}}\odot\cdots\odot\overline{v_{j_k}})
+(\overline{Av_{j_1}})\odot v_2\odot\cdots\odot \overline{v_{j_k}}
+\cdots +\overline{v_{j_1}}\odot\cdots\odot (\overline{Av_{j_k}})\nonumber \\
&= \left(-\frac{1}{2}\text{tr}(A)-i\sum_{l=1}^k s_{j_l}\right)(\overline{v_{j_1}}\odot\cdots\odot\overline{v_{j_k}}).\label{evec}\end{align}
Since the collection of vectors $\{\overline{v_{j_1}}\odot\cdots\odot \overline{v_{j_k}}\}$ with $j_i\in\{1,\cdots, d\}$ forms a basis of $S^k(\C^d)$, we have that
$$
\sigma(H_A|_{\H_k})=\left\{-\sum_{l=1}^k s_{j_l}\mid j_l\in\{1,\cdots, d\} \right\}+\frac{i}{2}\text{tr}(A).
$$

Thus theorem \ref{PCM} implies the expression of $\sigma_p(H_A)$ and $\sigma(H_A)$ given in the first part of theorem \ref{spec}. Moreover, equation \eqref{equi} shows that the polynomial $q$ given in theorem \ref{spec} is an eigenvector of $\hat B H_A\hat B^*$ with eigenvalue $\lb=-\sum s_j n_j+\frac{i}{2}\text{tr}(A)$, and therefore $\hat B^* q$ is an eigenvector of $H_A$ with the same eigenvalue. Furthermore, if $\varphi$ is an eigenvector, $P_k\varphi$ vanishes or it is an eigenvector of $H_A|_{\H_k}$, and this implies the last claim of the first part of theorem \ref{spec}.

 Let $Z(s)$ be the subgroup of $\R$ generated by the monoid $L(s)$ (i.e $Z(s)=L(s)-L(s)$). Alternatively, $Z(s)$ is the additive subgroup of $\R$ generated by $s_1,\cdots, s_d$. It is well known that every subgroup of $\R$ is either dense in $\R$ or it is of the form $x\mathbb Z$, for some $x\in\R$. If a) holds, since $L(s)\subseteq x\mathbb Z$, we have that there are $p_1,\cdots p_d\in \mathbb Z$ such that $s_j=p_jx$. Conversely, if $s_j=p_jx$ with $p_1,\cdots p_d\in \mathbb Z$ then $Z(s)=px\mathbb Z$, where  $p$ is the greatest common divisor of $|p_1|,\cdots,|p_d|$. This shows the equivalence between statements a) and b), and clearly a) implies c). Let us show that c) implies a). Let $\lambda\in L(s)$ and $r>0$ as in c). Let $D=(\lambda-r,\lb+r)-(\lambda-r,\lb+r)$. Then $D\backslash\{0\}$ is a nonempty open set and we claim that $D\backslash\{0\}\cap (L(s)-L(s))=\emptyset$. Let $x_0-y_0\in D\backslash\{0\}\cap (L(s)-L(s))$ with $x_0,y_0\in (\lambda-r,\lb+r)$ and $x_0\neq y_0$. Then there is $x,y\in L(s)$ such that $x_0-y_0=x-y$. Thus $x_0+y=x+y_0$, but the l.h.s. belongs $(\lambda +y-r,\lb+y+r)$ and the r.h.s. belongs $(\lambda +x-r,\lb+x+r)$. Hence $(\lambda +y-r,\lb+y+r)\cap (\lambda +x-r,\lb+x+r)\neq\emptyset$, which is a contradiction. Thus, $Z(s)$ is not dense and this implies a). The last claim follows directly from a).

\end{proof}
\begin{Remark}
 {\rm
\begin{enumerate}
    \item[i)]  If $s_1, \cdots s_d\in \mathbb Q$, then statement b) in theorem \ref{spec} holds. Indeed, if $s_j=\frac{p_j}{q_j}$ for each $j\in\{1,\cdots, d\}$, with $p_j\in\mathbb Z$ and $q_j\in\mathbb N$, then $Z(s)=\frac{p}{q}\mathbb Z$, where $q$ is the least common multiple of the denominators $q_1,\cdots, q_d$ and $p$ is the greatest common divisor of $\frac{q|p_1|}{q_1},\cdots, \frac{q|p_d|}{q_d}$ (see the argument in the proof of theorem \ref{smu}).
\item[ii)] When $A=\mathfrak J$, we have that $s_j=-1$ for every $1\leq j\leq d$, $H_A=H_0$ is the Harmonic Oscillator and our result asserts the well known identity $\sigma(H_0)=\mathbb N+\frac{d}{2}$. For each $i_0<j_0$, if $A=\big(\begin{smallmatrix} 
B & 0\\
0 & B
\end{smallmatrix}\big)$ with $b_{ij}=\delta_{i_0 i}\delta_{j_0 j}-\delta_{i_0 j}\delta_{j_0 i}$, then the corresponding eigenvalues are $\{0,-i,i\}$, $H_A=x_{j_0}\frac{\partial}{\partial x_{i_0}}-x_{i_0}\frac{\partial}{\partial x_{j_0}}$ is an angular momentum operator and theorem \ref{smu} asserts the well known identity $\sigma(H_A)=\mathbb Z$. 

\item[iii)] Theorem \ref{spec} implies  that the subgroup $Z(s)=L(s)-L(s)$ is dense in $\mathbb R$ if and only if $0\in \overline{Z(s)}$. Indeed,  if there are two sequences $(x_n),(y_n)$ in $N(s)$ such that $\lim( x_n-y_n)=0$, then $N(s)$ is not uniformly topologically discrete.
\end{enumerate}
  }   
\end{Remark}

\section{Multiplicity, spectral distribution and Weyl's law}\label{MyW}
In this section, we will  look for conditions to ensure that the spectrum of the operators studied in the previous section is discrete.

As we mentioned in the introduction, the operator $\mu(g)$ does not have discrete spectrum, for any $g\in \mathbb{U}(d)$. Indeed, if this was not the case then the eigenvalues of $g$ must be rational rotations, otherwise theorem \ref{smu} would imply that the eigenvalues of $\mu(g)$ would not be isolated.  Assume that  the eigenvalues of $g\in \mathbb{U}(d)$ are $\theta_1=\text{exp}(2\pi i\frac{p_1}{q_1}),\cdots, \theta_d=\text{exp}(2\pi i\frac{p_d}{q_d})$, with $p_1,\cdots p_d\in \mathbb Z$ and $q_1,\cdots q_d\in\mathbb N$. Equation \eqref{uevec} implies that for each $(n_1,\cdots, n_d)\in\mathbb N_0^d$ there are at least $\binom{M_1+n_1-1}{n_1}\cdots\binom{M_d+n_d-1}{n_d}$ linearly independent eigenvectors of $\mu(g)$ with eigenvalue $\lb=\overline{\theta_1}^{n_1}\cdots \overline{\theta_d}^{n_d}$, where $M_j$ is the multiplicity of the eigenvalue $\theta_j$. Fix $(n_1,\cdots, n_d)\in\mathbb N_0^d$. Then we have that $\overline{\theta_1}^{n_1}\cdots \overline{\theta_j}^{n_j}\cdots\overline{\theta_d}^{n_d}=\overline{\theta_1}^{n_1}\cdots \overline{\theta_j}^{n_j+tq_j}\cdots\overline{\theta_d}^{n_d}$, for every $t\in\mathbb N$. Therefore, the multiplicity of every eigenvalue of $\mu(g)$ is infinite.

Let us calculate the multiplicity of the eigenvalues of $H_A$. 
\begin{proof}[Proof of Proposition \ref{mul}]
For each $\lambda\in \sigma_p(H_A)$, let 

$$
M_\lb=\left\{(n_1,n_2,\cdots, n_d)\in \mathbb N^d_0\;\left|\; -\sum n_js_j+\frac{i}{2}\text{tr}(A)=\lb\right\}\right.
$$
We claim that $m_A(\lb)=\#( M_\lb)$. Indeed, since $H_A$ strongly commutes with the Harmonic Oscillator,  the eigenspace of $H_A$ corresponding to $\lb$ is the orthogonal sum over $k$ of the eigenspaces of $H_A|_{\H_k}$ corresponding to $\lb$. Furthermore, since each $T_k$ is an isomorphism, if we apply $T_k$ to the eigenvectors corresponding to $\lb$ given by equation \eqref{evec} we obtain a basis of the eigenspace of   $BH_AB^*|_{\mathfrak{F}^d_k}$ corresponding to $\lb$, and this implies our claim. In particular, $\lambda$ has finite multiplicity if and only if $M_\lb$ is finite.

Let us assume that $s_j=p_jx$, with $p_j\in\mathbb Z$ and suppose that there are $j_0$ and $j_1$ such that $s_{j_0}<0<s_{j_1}$. Without loss of generality, we can assume that $p_{j_0}<0<p_{j_1}$. Fix $n\in M_\lb$. Clearly, if we define $m_{j_0}=n_{j_0}+tp_{j_1}$, $m_{j_1}=n_{j_1}-tp_{j_0}$ and $m_j=n_j$ for $j\neq j_{0}, j_1$, then $m\in M_\lb$ for every $t\in\mathbb N$. Therefore, if $s_1,\cdots,s_d$ do not have the same sign, then $M_\lb$ is an infinite set. As we mentioned in the introduction, when $p_1,\cdots p_d$ have the same sign,  the expression of the multiplicity function $m_A$ follows from \cite[Theorem 2]{Wr} or references therein.
\end{proof}

A function admitting the type of expression described in proposition \ref{mul} for the multiplicity function $m_A$ is called a semi (or quasi)-polynomial of degree $d-1$ relative to modulus $d!$ (see \cite{Wr}). Notice that in the case $A=\mathfrak J$ this property is trivially satisfied; indeed, if $\lb=k+d/2$ then $m(\lb)=\binom{d+k-1}{k}$.  

For simplicity, in what follows we will assume that $s_j=p_j x <0$, and we will choose $x>0$.  

Before going into the proof of theorem \ref{dis}, let us recall the definition of an Ehrhart polynomial. Let $\mathcal{P}$ be a polyhedra having vertexes with integer coordinates and define 
\begin{equation}\label{Eh}
i(\mathcal P,k):=\#\left(k\mathcal P\bigcap\mathbb Z^d\right).
\end{equation}

E. Ehrhart showed that the map $i(\mathcal P,k)$ is a polynomial in $k$ (for instance, see \cite{Eh,St, Bri}), and this was the starting point of an important theory in combinatorics.

\begin{proof}[Proof of Theorem \ref{dis}]
From the proof of Proposition \ref{mul}, it is clear that the counting of eigenvalues function is given by 
$$
N_A(r)=\#\{\lb\in\sigma(H_A)\mid \lb\leq r\}=\#\left\{n\in\mathbb N^d_0\;\left|\; -\sum n_js_j+\frac{i}{2}\text{tr}(A)\leq r\right\}\right.
$$
Let $q$ be the minimal common multiple of $-p_1,-p_2,\cdots -p_d$ and let $\mathcal P$ be the polyhedra defined in the statement of Theorem \ref{dis}. Then the  vertexes of $\mathcal P$ are $V=\{-qp_j^{-1}\mid 1\leq j\leq d\}\cup\{0\}$. It is straightforward to show that
$$
i(\mathcal P,k)=N_A\left(kqx+\frac{i}{2}\text{tr}(A)\right).
$$
Taking $k=[q^{-1}x^{-1}(r-i\text{tr}(A)/2)]$, we obtain that $kqx+\frac{i}{2}\text{tr}(A)\leq r$ and the inequality in the statement of Theorem \ref{dis} follows. The value of the coefficients $c_d,c_{d-1}$ and $c_0$ is a well known fact from the theory of Ehrhart polynomials  (for instance, see subsection 1.3 in \cite{Bri}).
\end{proof}

Let us put back Planck's constant dependence on the operators studied before. More precisely, let us consider  the operators  $\mu_\hslash(g)$ and $\Op^\hslash(p_A)$ defined in subsection \ref{WQ}, for $g\in \mathbb{U}(d)$ and $A\in\mathfrak u(d)$. Since $\mu_\hslash(g)$ is unitary equivalent with $\mu(g)$ (see identity \eqref{equi-mu}), the spectrum of $\mu_\hslash(g)$ is described by Theorem \ref{smu}. Since $-id\mu_\hslash(A)$ is also unitary equivalent with $-id\mu(A)$, Proposition \ref{hb-H_A} and Theorem \ref{spec} imply the following result.

\begin{proposition}
For each $A\in\mathfrak u(d)$ and $\hslash >0$, we have that $\Op^\hslash(p_A)$ is unitary equivalent with $\hslash\Op(p_A)$. In particular, 
$$
\sigma_p[\Op^\hslash(p_A)]=\hslash L(s)+i\hslash\frac{\text{tr}(A)}{2}.
$$
\end{proposition}
\begin{Remark}
    {\rm The spectral analysis of the operator $H_A$ given in Theorem \ref{dis} and Proposition \ref{mul} remains exactly the same for the operator $H^\hslash_A:=\Op^\hslash(p_A)$. 
    }
\end{Remark}
Finally, let us prove the Weyl's law claimed in Theorem \ref{WL}.

\begin{Lemma}\label{PE}
Let $f:\R^d\to\R$ and $g:\R^{2d}\to\R$ given by $f(x)=\sum |s_j|x_j$ and $g(x,\xi)=2^{-1}\sum |s_j|(x_j^2+\xi_j^2)$ respectively. For each $t>0$, define $\mathcal P(t)=\{x\in\R^d\mid x\geq 0, f(x)\leq t\}$, $F(t)=\{x\in\R^d\mid x\geq 0, f(x)= t\}$ and $\mathcal E(t)=\{(x,\xi)\in\R^{2d}\mid g(x,\xi)\leq t\}$.  Then,
$$
|\mathcal P(t)|=(2\pi)^{-d}|\mathcal E(t)|=\frac{t^d}{d! |s_1|\cdot|s_2|\cdots|s_d|},
$$
\begin{equation}\label{face}
|F(t)|=|F(1)|t^{d-1}=\frac{\|s\|t^{d-1}}{(d-1)! s_1\cdots s_d}
\end{equation}
and
$$
|F(t)|=(2\pi)^{-d}\|s\|\int_{\partial\mathcal E(t)}\|\nabla g \|^{-1}d\mu_t,
$$
where $\mu_t$ is the measure corresponding to the canonical volume form on  $\partial\mathcal E(t)$.
\end{Lemma}
\begin{proof}
The first identity regarding the volume of the polyhedra $\mathcal P(t)$ and the ellipsoids $\mathcal E(t)$ are well known (and trivial). Clearly $\nabla f(x)=(s_1,\cdots ,s_d)$. Then coarea formula (for instance, see theorem 3.2.12 in \cite{Fe}) implies that
$$
|\mathcal P(t)|=\int_0^t \left(\int_{F(r)}\|\nabla f\|^{-1} d\eta_r\right ) dr=\|s\|^{-1}\int_0^t |F(r)| dr,
$$
where $\eta_r$ is the measure corresponding to the canonical volume form on the $(d-1)$-face $F(r)$. Derivating the previous equality with respect to $t$ leads to equation \eqref{face}.  Similarly,  we have that
$$
|\mathcal E(t)|=\int_0^t \left(\int_{\partial\mathcal E(r)}\|\nabla g\|^{-1} d\mu_r\right ) dr
$$
and derivating with respect to $t$ we obtain the second identity.
\end{proof}
\begin{proof}[Proof of Theorem \ref{WL}]
The inequality established in Theorem \ref{dis} implies that in order to obtain the asymptotic growth claimed in Theorem \ref{WL}, we can assume that $(q x)^{-1}(\frac{r}{\hslash}-\frac{i}{2}\text{tr}(A))$ belongs to $\mathbb{N}$. Then,

\begin{align*}\label{hdis}
N^\hslash_A(r)=& \sum_{j=1}^d c_j\left[(q x)^{-1}(\frac{r}{\hslash}-\frac{i}{2}\text{tr}(A))\right]^j \\
=&c_d ( q x)^{-d}(\hslash^{-1}r)^d+(c_{d-1} ( q x)^{-d+1}-\frac{i}{2}dc_d ( q x)^{-d}\text{tr}(A))(\hslash^{-1}r)^{d-1}\\
&+O\left((\hslash^{-1}r)^{d-2}\right)
\end{align*}

Since $c_d=|\mathcal P(qx)|$, $i\text{tr}(A)=\sum_{j=1}^d|s_j|$ and 

$$
c_{d-1}=\frac{1}{2}|\partial\mathcal P(qx)|=\frac{(qx)^{d-1}}{2(d-1)!}\sum_{j=1}^d\frac{|s_j|}{|s_1|\cdot|s_2|\cdots|s_d|}+\frac{1}{2}|F(qx)|,
$$
the previous lemma implies that 
\begin{align*}
N^\hslash_A(r)&=(2\pi\hslash)^{-d}|\mathcal E(r)|+\frac{\hslash^{-d+1}}{2}|F(r)|+O\left((\hslash^{-1}r)^{d-2}\right)\\
&=(2\pi)^{-d}\hslash^{-d}|\mathcal E(r)|+\frac{\|s\|}{2}(2\pi)^{-d}
\hslash^{-d+1}\int_{\partial\mathcal E(t)}\|\nabla g \|^{-1}d\mu_t+O\left((\hslash^{-1}r)^{d-2}\right)
\end{align*}

Since $A\in\mathfrak{u}(d)$, there is $U\in\mathbb{U}(d)$ such that $UAU^*=D$, where $D$ is the diagonal matrix with entrances $is_1,is_2,\cdots is_d$. Recall that, under the identification of $\R^{2d}$ with $\C^d$, the Lie group $\mathbb U(d)$ corresponds with  $\mathbb{O}(2d)\cap Sp(d)$. Thus,
$$
p_A(w)=-\frac{1}{2}\langle w,A\mathfrak{J}w\rangle=-\frac{1}{2}\langle w,U^*DU\mathfrak{J}w\rangle=-\frac{1}{2}\langle Uw,\mathfrak{J}D U w\rangle=g(Uw).
$$
Since $U$ preserves volume, we have that
$$
|\mathcal E(t)|=|\{(x,\xi)\in\R^{2d}\mid p_A(x,\xi)\leq t\}|.
$$
Applying coarea formula, we obtain that
$$
\int_0^t\left(\int_{\partial\mathcal E(r)}\|\nabla g\|^{-1}d\mu_r\right) dr = \int_0^t\left(\int_{\{(x,\xi)\mid p_A(x,\xi)=r\}}\|\nabla p_A\|^{-1}d\mu_r^{A}\right) dr,
$$ 
where $\mu_t^{A}$ is the measure corresponding to the canonical volume form on  $\{(x,\xi)\mid p_A(x,\xi)=r\}$. Derivating both sides of the previous identity, we have that 
$$
\int_{\partial\mathcal E(r)}\|\nabla g\|^{-1} = \int_{\{(x,\xi)\mid p_A(x,\xi)=r\}}\|\nabla p_A\|^{-1}
$$ 
and this finishes the proof.
\end{proof}


\end{document}